\documentclass{elsarticle}

\usepackage{fullpage}
\usepackage{amsmath}
\usepackage{amsthm}
\usepackage{amsfonts}
\usepackage{amssymb}
\usepackage{enumerate}
\usepackage{mathrsfs}
\usepackage{mathtools}
\usepackage{tikz}

\usetikzlibrary{arrows}
\usetikzlibrary{decorations.pathreplacing}

 \def\ikg{\mathcal{I}^k_G}
 \def\ikp{\mathcal{I}^k_P}

\newtheorem{construction}{Construction}[section]
\newtheorem{theorem}[construction]{Theorem}

\newtheorem{conjecture} {Conjecture}
\newtheorem{corollary} [construction]{Corollary}

\newtheorem{definition} [construction]{Definition}

\newtheorem{lemma} [construction]{Lemma}

  \newcommand{\flip}{\ensuremath{\mbox{\sf flip}}}

\begin{document}
\begin{frontmatter}

\title{On Stars in Caterpillars and Lobsters}

\address[daj]{Instituto de Matem\'{a}tica Aplicadas de San Luis, CONICET-UNSL.}
\address[UNSL]{Departamento de Matem\'atica, Universidad Nacional del San Luis, San Luis, Argentina}

\author[UNSL]{Emiliano J.J. Estrugo}\ead{juan.estrugo.tag@gmail.com} 
\author[daj,UNSL]{Adri\'{a}n Pastine}\ead{agpastine@unsl.edu.ar}

\begin{abstract} 
The family of all $k$-independent sets of a graph containing a fixed vertex $v$ is called a {star} and $v$ is called its center. Stars are interesting for their relation to Erd\"{o}s-Ko-Rado graphs. Hurlbert and Kamat conjectured that in trees the largest stars are centered in leafs. This conjecture was disproven independently by Baber, Borg, and Feghali, Johnson, and Thomas. In this paper we introduce a tool to bound the size of stars centered at certain vertices by stars centered at leafs. We use this tool to show that caterpillars and sunlet graphs satisfy Hurlbert and Kamat's conjecture, and to show that the centers of the largest stars in lobsters are either leafs or spinal vertices of degree 2.
\end{abstract}

\end{frontmatter}

\section{Introduction}
  The vertex and edge sets of $G$ are denoted by $V (G)$ and $E(G)$, respectively. If $H$ is a graph such that $V(H) \subseteq V(G)$ and $E(H) \subseteq E(G)$, then we say that $G$ contains $H$ or $H$ is a subgraph of $G$ and denote it $H\subseteq G$. Given a graph $G$ and a set of vertices $W$, by $G-W$ we denote the subgraph obtained by removing from $G$ every vertex in $W$ and the corresponding edges.  Also, given a vertex $v\in V(G)$ the degree of $v$ is the number of vertices adjacent to $v$ and denoted by $\deg(v)$.
  
     For a positive integer $n$, let $[n]:=\{1,2,\dots,n\}$. If $v_1 , v_2 , \dots , v_n$ are the distinct vertices of a graph $G$ with $E(G) = \{v_i v_{i + 1} : i\in [n-1]\}$, then $G$ is called a $(v_1,v_n)$-path or simply a path. We usually denote paths with the letter $P$.
  
  A graph $G$ is a tree if $|V(G)| \geq 2$ and $G$ contains exactly one $(v,w)$-path for every $v,w\in V(G)$ with $v\neq w$. 
  A vertex $v$ of $G$ is called a pendent vertex if it has only one vertex adjacent, i.e. $\deg(v)=1$. Pendent vertices of trees are also called \emph{leaves}.

   A subset $I$ of $V(G)$ is an independent set of $G$ if $vw \notin E(G)$ for every $v,w\in I$. Let $\mathcal{I}_G$ denote the family of all independent sets of $G$, and $\mathcal{I}^k_G$ denote the family of all independent sets of $G$ of size $k$. For $v \in V(G)$ the family $\ikg (v):= \{ A\in \ikg:  v\in A \}$ is called a \emph{star} of $\ikg$ and $v$ is called its \emph{star center}. 
   
   The study of star centers in graphs is related to the study of Erd\"{o}s-Ko-Rado graphs. 
   A graph $G$ is said to be $k$-EKR (Erd\"os-Ko-Rado) if for any family of independent sets $\mathcal{F}\subset \mathcal{I}^k_G$ satisfying $A\cap B\neq \emptyset$ for every $A,B\in \mathcal{F}$, there is a vertex $x\in V(G)$ such
   that $|\mathcal{F}|\leq \ikg (v)$. Studying this problem Holroyd and Talbot, \cite{Holroyd1} made the following conjecture.
\begin{conjecture}\label{conjekr}
Let $G$ be a graph, and let $\mu(G)$ be the size of its smallest maximal independent set.
Then $G$ is $k$-EKR for every $1\leq k \leq \mu(G)/2$.
\end{conjecture}   
Most of the graphs known to satisfy Conjecture \ref{conjekr} contain at least one isolated vertex, see  \cite{Holroyd2,Borg2,Hurlbert1}.
This is because in order to prove that a graph is $k$-EKR one usually has to find the center of the largest star,
which is trivial when there is an isolated vertex, but can be quite difficult otherwise.
   
   This is why in \cite{Hurlbert1} Hurlbert and Kamat studied  stars in trees, as a first step
   to study Conjecture \ref{conjekr}. There they  conjectured that for any $k \geq 1$ and any tree $T$, there exists a leaf $l\in V(T)$, such that $I^k_T(l)$ is a star of $\mathcal{I}^k_T$ of maximum size.

 \begin{conjecture}\label{conjearbol}
 	For any $k \geq 1$ and any tree $T$, there exists a leaf $l$ of $T$ such that $|\mathcal{I}^k_T(v)| \leq |\mathcal{I}^k_T(l)|$ for
    each $v\in V(T)$.
 \end{conjecture}
Hurlbert and Kamat proved this conjecture for $k \leq 4$, but the conjecture was shown to be false independently by, Baber \cite{Baber}, Borg \cite{Borg}, and Feghali, Johnson and Thomas \cite{Feghali}. They all arrived at the same family of counter examples, given by the tree $T_m$ that consists of
a set of $m$ paths of length $2$ connected to a vertex $v_1$, another set of $m$ paths of length $2$ connected to a vertex $v_2$, and a vertex $v_0$ connected to both $v_1$ and $v_2$, see Figure \ref{figlob}. In \cite{Baber,Borg,Feghali} the authors showed that for $m \geq 3$, the vertex $v_0$ of $T_m$ satisfies that and $|\mathcal{I}^k_T(l)| < |\mathcal{I}^k_T(v_0)|$ for any leaf $l$ of $T_m$ and any $5 \leq k \leq 2m + 1$ thus giving a counterexample to the Conjecture. 

\begin{figure}[h]
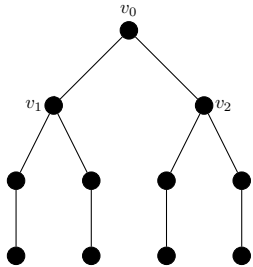
\label{figlob}
	\centering
	\tikz
{
	\node[circle,scale=.7] at (3,0.26) (v) {$v_0$};
	\node[fill,circle,draw,scale=.7] at (3,0) (1) {};
	\node[circle,scale=.7] at (1.74,-1) (v) {$v_1$};
	\node[fill,shape=circle,draw,scale=.7] at (2,-1) (2) {};
\node[circle,scale=.7] at (4.26,-1) (v) {$v_2$};
	\node[fill,shape=circle,draw,scale=.7] at (4,-1) (3) {};
	\node[fill,shape=circle,draw,scale=.7] at (1.5,-2) (4) {};
	\node[fill,shape=circle,draw,scale=.7] at (2.5,-2) (5) {};
	\node[fill,shape=circle,draw,scale=.7] at (3.5,-2) (6) {};
	\node[fill,shape=circle,draw,scale=.7] at (4.5,-2) (7) {};
	\node[fill,shape=circle,draw,scale=.7] at (1.5,-3) (8) {};
	\node[fill,shape=circle,draw,scale=.7] at (2.5,-3) (9) {};
	\node[fill,shape=circle,draw,scale=.7] at (3.5,-3) (10) {};
	\node[fill,shape=circle,draw,scale=.7] at (4.5,-3) (11) {};
	\draw (1) [scale=1] to (2);
	\draw (1) [scale=1] to (3);
	\draw (2) [scale=1] to (4);
	\draw (2) [scale=1] to (5);
	\draw (3) [scale=1] to (6);
	\draw (3) [scale=1] to (7);
	\draw (2) [scale=1] to (2);
	\draw (4) [scale=1] to (8);
	\draw (5) [scale=1] to (9);
	\draw (6) [scale=1] to (10);
	\draw (7) [scale=1] to (11);
}
\caption{The graph $T_2$, whose largest k-star for $k\geq 5$ is centered at $v_0$.}
\end{figure}

We say that a graph $G$ satifies $HK$ if for any $k \geq 1$, there exists a pendent vertex $l$ of $G$ such that 
$|\mathcal{I}^k_G(v)| \leq |\mathcal{I}^k_G(l)|$ for
each $v\in V(G)$.
In \cite{Hurlbert2}, Hurlbert and Kamat considered spiders, which are trees that have exactly one vertex of degree greater than $2$. They use a function on the paths of the spider called ``flip'' to prove that spiders satisfy $HK$. They then used flips together with a different function called ``switch'' to order the size of the stars at the 
different leaves based on the length and parity of the path from the leaf to the vertex of maximum degree.
This is the only result so far showing that a family of graphs satisfy $HK$.

In this manuscript we study other families of graphs that satisfy $HK$.
A tree $C$ is a {caterpillar} if $G$ removing the leaves and incident edges produces a path graph $P$, called the spine. A tree $L$
is called a {lobster} if removing the leaves and incident edges produces a caterpillar $C$. 
Notice that the graph $T_m$ studied in \cite{Baber,Borg,Feghali} is a lobster. Thus we know that lobsters do not necessarily satisfy $HK$, and raises the question of whether caterpillars do satisfy it. We answer this question
by improving on the ``flip'' technique from \cite{Hurlbert2} to bound the size of the stars centered at some vertices by the size of stars centered at pendent vertices. We also use this technique to show that if the star centered at a vertex $v$ is larger than the stars centered at any leaf $\ell$, then $\deg(v)=2$ and $v$ is in the spine of the caterpillar obtained by removing the leafs of $L$.
We also show that the technique is more general, by mentioning another family of graphs that satisfy $HK$.

The rest of this paper is organized as follows. In Section 2 we introduce our technique. In Section 3 we apply the technique to show that some families of graphs, including caterpillars, are $HK$, and to study the possible centers of stars in lobsters. In Section 4 we give some concluding remarks.
\section{On the flipping technique}
We start this section with a quite straightforward result, showing that ``flipping'' a path $P$ gives a bijection from the family of independent sets of $P$ containing one leaf of the path to the family of independent sets of $P$ containing the other leaf. We then extend this idea to flipping a path $P$ contained in a bigger graph $G$, assuming that $P$ satisfy certain conditions, and show that this gives an injection from the family of independent sets of $G$ containing
a vertex $v$, to the family of independent sets of $G$ containing a pendent vertex $l$. Our technique ends up being a corollary of this last result.

\begin{definition}
	Let $\flip: V(P)\rightarrow V(P)$ be defined by $\flip(v) = n+1-v$. 
\end{definition}
Notice that $\flip$ is its own inverse function, i.e. $\flip ^2=id_{V(P)}$.

Considering $\flip$ as function from $\ikp(1)$ to $f(V(P))$, by mapping each set $A\subseteq V(P)$ to $f(A)$, we have the following straightforward lemma.

\begin{lemma}\label{lafuncion}
The function $\flip$ maps independent sets into independent sets, and induces a bijection from $\mathcal{I}^k_P$ onto itself. Furthermore,  $\flip(\mathcal{I}_P(1))= \mathcal{I}_P(n)$. 
\end{lemma}
\begin{proof}
Let $A$ be an independent set and assume $v,w\in \flip(A)$, with $v\neq w$. Then $v=n+1-x$ and $w=n+1-y$, where $x,y\in A$. Hence $x\neq y$ and $|v-w|=|x-y|>1$. Thus, $\flip(A)$ is an independent set, and $\flip$ maps independent sets 
into independent sets.
Hence $\flip$ induces a bijection from $\mathcal{I}^k_P$ onto itself, because it is a bijection on $V(P)$.

Finally, notice that for every $A\in \ikp(1)$, $1\in A$ and $A$ is an independent $k$-set. Thus, $\flip(1)=n\in \flip(A)$ and $\flip (A)$ is an independent set containing $n$.
\end{proof}

We are going to extend the $\flip$ function to more general graphs, but for this we need to focus on a special type of paths.
\begin{definition}
Let $G$ be a graph and $P=v_1v_2,\dots v_n$ a path  of length $n$ such that $P\subset G$. We say that $P$ is an {escape path from $v_1$ to $v_n$ in G} if $\deg(v_n)=1$ and $\deg(v_i)=2$ for every $2\leq i \leq n-2$.
If this is the case we say that {$v_1$ has an escape path to $v_n$}.
\end{definition}
Notice that if $P$ is an escape path in $G$, then  $G$ is obtained by joining a graph $G_1$ to the first vertex of the path $P$, a graph $G_2$ to the second to last
vertex of $P$, and a graph $G_3$ to both vertices, (see Figure \ref{fig2}).

\begin{figure}
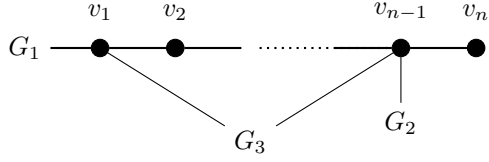

	\centering
	\tikz
	{
		\node[] at (0,0) (g1) {$G_1$};
		\node[fill,shape=circle,draw,scale=.7] at (1,0) (1) {};
		\node[shape=rectangle,above=7pt] at (1) (11) {$v_1$};			
		\node[fill,shape=circle,draw,scale=.7] at (2,0) (2) {};
		\node[shape=circle,scale=.7] at (3,0) (a) {};
		\node[shape=rectangle,above=7pt] at (2) (22) {$v_2$};
		\node[shape=circle,scale=.7] at (4,0) (b) {};
		\node[fill,shape=circle,draw,scale=.7] at (5,0) (n) {};
		\node[shape=rectangle,above=7pt] at (n) (nn) {$v_{n-1}$};	
		\node[fill,shape=circle,draw,scale=.7] at (6,0) (n1) {};
		\node[shape=rectangle,above=7pt] at (n1) (n1n1) {$v_n$};
		\node[] at (5,-1) (g2) {$G_2$};								
		\node[] at (3,-1.25) (g3) {$G_3$};		
		\draw (1) [scale=1,thick] to (2);
		\draw (2) [scale=1,thick] to (a);
		\draw (b) [scale=1,thick] to (n);
		\draw (g1) [scale=1,thick] to (1);
		\draw (a) [dotted,scale=1,thick] to (n);
		\draw (n) [scale=1,thick] to (n1);
		\draw (n) [scale=1] to (g2);
		\draw (n) [scale=1] to (g3);
		\draw (1) [scale=1] to (g3);
	}
	\caption{A graph $G$ containing an escape path $P$.}
	\label{fig2}
\end{figure}

\begin{definition}
Let $G$ be a graph and $P=v_1v_2,\dots v_n$ be an escape path from $v_1$ to $v_n$ in $G$. Then the {flip of $P$ in $G$}, $\flip_P:V(G)\rightarrow V(G)$, is the function defined as follows
		\[
	\flip_P(v) = \begin{cases}
	v & \text{if $v\notin V(P)$}\\
	v_{n+1-i} & \text{if $v=v_i\in V(P)$}.
	\end{cases}
	\]
\end{definition}
The flip of $P$ acts as the function $\flip$ on the vertices of $P$, and leave the vertices outside of $P$ fixed.

Again it is easy to see that $\flip_P^2(v)=v$ for every vertex $v$. On the other hand,
 $\flip_P$ does not necessarily map independent set to independent sets. As an example take the graph in Figure \ref{fig2}. If $A$ is an independent set containing $v_n$ and a vertex $v\in \left(V(G_1)\cup V(G_3)\right)$ adjacent to $v_1$, then
 $\flip_P(A)$ is not independent as it contains $v$ and $v_1$. Similarly, if $B$ is an independent set containing $v_2$ and a vertex $w\in \left(V(G_2)\cup V(G_3)\right)$ adjacent to $v_{n-1}$, then $\flip_P(B)$ is not independent as it contains
 both $w$ and $v_2$. The next result shows that  $\flip_P$ induces a one to one mapping of independent sets containing $v_1$ into independent sets containing $v_n$.
\begin{lemma}\label{laFuncion}
Let $G$ be a graph. If $P=v_1,v_2,\dots , v_n$ is an escape path from $v_1$ to $v_n$ in $G$, then $\flip_P$ 
induces a one to one mapping from $\mathcal{I}_G(v_1)$ into $\mathcal{I}_G(v_n)$.
\end{lemma}
\begin{proof}
Let $A\in \mathcal{I}_G(v_1)$, and consider $\flip_P(A)$.

As $\flip_P(v_1)=v_{n+1-1}=v_n$, $v_n\in \flip_P(A)$. We need to show that  $\flip_P(A)$ is an independent set.

As $\flip_P$ acts as $\flip$ on the vertices of $P$, $\flip_P(A)\cap V(P)$ is an independent set by Lemma \ref{lafuncion}.
As $\flip_P$ fixes the vertices in $G- V(P)$, $\flip_P(A)\cap V\left(G- V(P)\right)$ is an independent set.
Then, the only way for $\flip_P(A)$ to have adjacent vertices is for one of the vertices to be in $P$ and the other to be in $G- V(P)$. 
But the only vertices in $P$ having neighbors outside of $P$ are vertices $v_1$ and $v_{n-1}$.

The set $A$ does not contain any vertices in $G- V(P)$ that are neighbors with $v_1$, because $A$ is independent. Thus, as $\flip_P$ fixes vertices in $G- V(P)$, $\flip_P(A)$ does not contain vertices in $P$ that are adjacent to $v_1$.
Furthermore, the independence of $A$ also ensures that $v_2\not\in A$ and thus $\flip_P(A)$ does not contain $v_{n-1}$.
Therefore, no pair of vertices in $\flip_P(A)$ can be adjacent.
Hence, $\flip_P$ maps $\mathcal{I}_G(v_1)$ into $\mathcal{I}_G(v_n)$.

But, $\flip_P^2 (v)=v$ for every vertex $v$, we have that $\flip_P^2(B)=B$ for every set $B$, hence $\flip_P$ induces a one to one mapping on sets of vertices.

Therefore,   $\flip_P$ 
induces a one to one mapping from $\mathcal{I}_G(v_1)$ into $\mathcal{I}_G(v_n)$.
\end{proof}
In particular, Lemma \ref{laFuncion} implies that  $|\ikg(v_1)| \leq |\ikg(v_n)|$, yielding.
\begin{theorem}\label{teogeneral}
Let $G$ be a graph, $v\in V(G)$, and $k\geq 1$. If there is an escape path from $v$ to a pendent vertex $\ell$, then $|\ikg(v)| \leq |\ikg(l)|$.
\end{theorem}
\begin{proof}
The result follows from Lemma \ref{laFuncion}, as $\flip_P$ induces a one to one mapping from $\mathcal{I}_G(v)$ into $\mathcal{I}_G(\ell)$.
\end{proof}
Theorem \ref{teogeneral} proves to be an important tool in the study of stars in graphs, as we show in the next section.
\section{Star Centers for some Graphs and Trees}
In this section we study the implications of Theorem \ref{teogeneral} in different families of graphs.
We begin by proving that spiders and caterpillars are HK. Next we use Theorem \ref{teogeneral} to study lobsters.
We end the section by showing a family of graphs that satisfy $HK$ but are not trees.
Note that the graphs in Theorem \ref{teogeneral} are in a sense quite generic, so the idea is to identify paths whose end vertices are pendent vertices thus the stars centered on vertices other than the pendent vertex, are smaller than the star centered on the pendent vertex.

\subsection{Spiders}
Given that our technique is based on the technique used in \cite{Hurlbert2}, it is an unsurprising corollary of Theorem \ref{teogeneral} that spiders satisfy HK. To see this notice that every vertex of degree $2$ has an escape path to at least one leaf, whereas the vertex of degree greater than $2$ has escape paths to every leaf.
\begin{corollary}
Spiders satisfy HK.
\end{corollary}
\subsection{Caterpillar}
It is easy to see that every vertex $v$ of a caterpillar $G$ that is not a leaf has an escape path to a leaf. To see this just take the leaf $\ell$ that is closest to $v$. Then if $w$ is vertex in the path from $v$ to $\ell$ that is not adjacent to $\ell$, $w$ cannot be adjacent to a leaf, and thus $\deg(w)=2$. 
Hence we have the following.
\begin{corollary}\label{cater}
Caterpillars satisfy HK.
\end{corollary}

\subsection{Lobster Graphs}
In general, Theorem \ref{teogeneral} gives some insight into which vertices may be the center of the largest stars.
This is specially true for lobsters. Let $L$ be a lobster, $C$ the caterpillar obtained by removing the leaves of $L$, and $P$ the path  obtained by removing the leaves of $C$. If $v\in V(P)$ we say that $v$ is a spinal vertex. Clearly every vertex in $L-V(C)$ is a leaf. Every vertex in $L-V(P)$ is either a leaf, or adjacent to a leaf and thus has an escape path to a leaf.  For $v\in V(P)$,  if $v$ has degree at least $3$ in $L$, then $v$ has a neighbor $w\not\in V(P)$. As $L$ is a lobster, the vertex $w$ is either a leaf, or is neighbor with a leaf $\ell$. In either case, there is an escape path from $v$ to a leaf.
Thus the only vertices that may not have escape paths are the vertices in $V(P)$ of degree $2$. Hence, Theorem \ref{teogeneral} implies that the largest stars are centered in leaves or in vertices on the path $P$ of degree $2$.

\begin{theorem}\label{teolobs}
Let $L$ be a lobster, $v\in V(L)$ and $k\geq 1$. If $v$ is not a spinal vertex of $L$ or $\deg(v)\neq 2$, then there is a leaf $\ell \in V(L)$ such that $\mathcal{I}_L^k(\ell)\geq \mathcal{I}_L^k(v)$. 
\end{theorem}

In some sense, Theorem \ref{teolobs} shows that lobsters are as close as possible to satisfy HK. This is because the degrees of the centers of the largest stars are either $1$ or $2$.

\subsection{Sunlet graphs}
Theorem \ref{teogeneral} can be applied to show that many different graphs, not necessarily trees, satisfy $HK$.
To illustrate this, we study sunlet graphs. 
The $n$-sunlet graph is the graph on $2n$ vertices obtained by attaching $n$ pendant edges to a cycle graph $C_n$. Also attaching $n$ paths of different lengths.
Again, it is easy t see that every non pendent vertex has an escape path to a pendent vertex, yielding the following result.
\begin{corollary}
Sunlet graphs satisfy HK.
\end{corollary}

\section{Conclusion}
In this paper we gave a tool to reduce the number of vertices one has to study in order to find the center of the largest stars. We used this to prove that the center of the largest stars in caterpillars and sunlet graphs are leaves,
and that the center of the largest stars in lobster graphs are either leaves or spine vertices of degree 2.
A next step to prove that any of this families is EKR is to decide among leaves (and spine vertices of degree 2), which 
one is the center of the largest star.
\section{Acknowledgements}
This work was partially supported by the Universidad Nacional de San Luis, grant PROICO 03-0918, 
and MATH AmSud, grant 18-MATH-01. 
The first author was supported by a doctoral scholarship from 
Consejo Nacional de Investigaciones Cient\'{i}ficas y 
T\'{e}cnicas (CONICET).


\begin{thebibliography}{99}
\bibitem{Baber}
R. Baber, 
Some results in extremal combinatorics, PhD diss., UCL (University College London), 2011.

\bibitem{Borg}
P. Borg, Stars on Trees, {\em Discrete Math.} {\bf 340} 
(2017), pp. 1046--1049.

\bibitem{Borg2}
P. Borg and F.C. Holroyd, The Erd\"{o}s-Ko-Rado properties of various graphs containing singletons, {\em Discrete Math.} {\bf 309} 
(2009), 2877-2885.

\bibitem{Feghali}
Feghali C., Johnson M., Thomas D.,
Erd\"os-Ko-Rado theorems for a family of trees
{\em Discrete Appl. Math.}, 236 (2018), pp. 464--471.

\bibitem{Holroyd2}
F.C. Holroyd, C. Spencer, and J. Talbot, Compression and Erd\"{o}s-Ko-Rado
Graphs,
{\em Discrete Math.} 293 (2005), no. 1-3, 155--164.

\bibitem{Holroyd1}
F.C. Holroyd, J. Talbot, Graphs with the Erdo\"{o}s-Ko-Rado property, 
{\em Discrete Math.} 293 (1-3) (2005) 165--176.

\bibitem{Hurlbert1}
Hurlbert G., Kamat V.,
Erd\"os-Ko-Rado theorems for chordal graphs and trees,
{\em J. Combin. Theory Ser. A}, 118 (2011), 829--841.

\bibitem{Hurlbert2}
Hurlbert G., Kamat V.,
On intersecting families of independent sets in
trees, 
{\em arXiv preprint}, arXiv:1610.08153 (2016)..
\end{thebibliography}
\end{document}